\documentclass[12pt]{article}
\usepackage{amsthm,amssymb,amsmath}

\advance\textheight10pt

\newcommand\cL{\mathcal L}
\newcommand\cP{\mathcal P}
\newcommand\F{\mathbb{F}}
\newcommand\cU{\mathcal U}
\newcommand\cl{\mathord{\mkern1mu:\mkern2mu}}

\newtheorem{theorem}{Theorem}
\newtheorem{lemma}[theorem]{Lemma}

\newtheorem{corollary}[theorem]{Corollary}
\newtheorem{conjecture}[theorem]{Conjecture}

\newcommand\Lm[1]{Lemma~\ref{l:#1}}
\newcommand\Th[1]{Theorem~\ref{t:#1}}
\newenvironment{lst}{\begin{list}{$\bullet$}{\topsep=0pt\itemsep=0pt}}{\end{list}}

\newdimen\unit\newdimen\psep\newcount\nd\newcount\ndx\newbox\dotb\newbox\ptbox
\newdimen\dx\newdimen\dy\newdimen\dxx\newdimen\dyy\newdimen\hgt
\newdimen\xoff\newdimen\yoff
\newcommand\clap[1]{\hbox to 0pt{\hss{#1}\hss}}
\newcommand\vdisk[1]{{\font\dotf=cmr10 scaled #1\dotf.}}
\newcommand\varline[2]{\setbox\dotb\hbox{\vdisk{#1}}\xoff=-.5\wd\dotb
\wd\dotb=0pt\yoff=-.5\ht\dotb\psep=#2\ht\dotb}
\newcommand\varpt[1]{\setbox\ptbox\clap{\vdisk{#1}}\setbox\ptbox
\hbox{\raise-.5\ht\ptbox\box\ptbox}}
\newcommand\cpt{\copy\ptbox}
\newcommand\point[3]{\rlap{\kern#1\unit\raise#2\unit\hbox{#3}}}
\newcommand\setnd[4]{\dx=#3\unit\advance\dx-#1\unit\divide\dx by\psep
\dy=#4\unit\advance\dy-#2\unit\divide\dy by\psep
\multiply\dx by\dx\multiply\dy by\dy\advance\dx\dy\nd=1\advance\dx-1sp
\loop\ifnum\dx>0\advance\dx-\nd sp\advance\nd1\advance\dx-\nd sp\repeat}
\newcommand\dl[4]{{\setnd{#1}{#2}{#3}{#4}\dline{#1}{#2}{#3}{#4}\nd}}
\newcommand\dline[5]{{\nd=#5\hgt=#2\unit\dx=#3\unit\advance\dx-#1\unit
\divide\dx by\nd\dy=#4\unit\advance\dy-#2\unit\divide\dy by\nd
\advance\hgt\yoff\rlap{\kern#1\unit\kern\xoff\loop\ifnum\nd>1\advance\nd-1
\advance\hgt\dy\kern\dx\raise\hgt\copy\dotb\repeat}}}

\newcommand\qellip[4]{{\setnd{0}{0}{#3}{#4}\dx=\unit\dy=0pt\raise\yoff\rlap{%
\kern#1\unit\kern\xoff\raise#2\unit\hbox{\loop\ifnum\dx>0\rlap{\kern#3\dx
\raise#4\dy\copy\dotb}\hgt=\dx\divide\hgt by\nd\advance\dy\hgt\hgt=\dy
\divide\hgt by\nd\advance\dx-\hgt\repeat\rlap{\raise#4\dy\copy\dotb}}}}}
\newcommand\bez[6]{{\setnd{#1}{#2}{#3}{#4}\ndx=\nd\setnd{#3}{#4}{#5}{#6}
\ifnum\ndx>\nd\nd=\ndx\fi\dx=#3\unit\advance\dx-#1\unit\dy=#4\unit
\advance\dy-#2\unit\dxx=#5\unit\advance\dxx-#1\unit\dyy=#6\unit\advance
\dyy-#2\unit\advance\dxx-2\dx\advance\dyy-2\dy\divide\dxx by\nd\divide\dyy
by\nd\advance\dx.25\dxx\advance\dy.25\dyy\divide\dx by\nd\divide\dy by\nd
\multiply\nd by2\dx=100\dx\dy=100\dy\dxx=100\dxx\dyy=100\dyy\divide\dxx by\nd
\divide\dyy by\nd\hgt=#2\unit\raise\yoff\rlap{\kern#1\unit\kern\xoff
\raise\hgt\copy\dotb\loop\ifnum\nd>0\advance\nd-1\advance\hgt0.01\dy
\kern0.01\dx\raise\hgt\copy\dotb\advance\dx\dxx\advance\dy\dyy\repeat}}}
\newcommand\ptu[3]{\point{#1}{#2}{\cpt\raise.8ex\clap{$\scriptstyle{#3}$}}}
\newcommand\ptd[3]{\point{#1}{#2}{\cpt\raise-1.6ex\clap{$\scriptstyle{#3}$}}}
\newcommand\ptr[3]{\point{#1}{#2}{\cpt\raise-.4ex\rlap{$\ \scriptstyle{#3}$}}}
\newcommand\ptl[3]{\point{#1}{#2}{\cpt\raise-.4ex\llap{$\scriptstyle{#3}\ $}}}
\newcommand\ptlu[3]{\point{#1}{#2}{\raise.6ex\clap{$\scriptstyle{#3}$}}}
\newcommand\ptld[3]{\point{#1}{#2}{\raise-1.6ex\clap{$\scriptstyle{#3}$}}}
\newcommand\ptlr[3]{\point{#1}{#2}{\raise-.4ex\rlap{$\,\scriptstyle{#3}$}}}
\newcommand\ptll[3]{\point{#1}{#2}{\raise-.4ex\llap{$\scriptstyle{#3}\,\,$}}}
\newcommand\pt[2]{\point{#1}{#2}{\cpt}}
\newcommand\px[2]{\point{#1}{#2}{\raise-2.5pt\clap{$\star$}}}

\newcommand\thnline{\varline{400}{.3}}

\varpt{3500}\thnline\unit=2em
\title{Minimal symmetric differences of lines in projective planes}
\author{Paul Balister%
\thanks{Department of Mathematical Sciences, University of Memphis, Memphis TN 38152, USA}
\and B\'ela Bollob\'as%
\thanks{Department of Pure Mathematics and Mathematical Statistics,
Wilberforce Road, Cambridge CB3 0WB, UK, and
Department of Mathematical Sciences, University of Memphis, Memphis TN 38152, USA}
\and Zolt\'an F\"uredi%
\thanks{Alfr\'ed R\'enyi Institute of Mathematics, 13--15 Re\'altanoda Street, 1053 Budapest, Hungary.
Research supported in part by the Hungarian National Science Foundation OTKA 104343,
and by the European Research Council Advanced Investigators Grant 267195.}
\and John Thompson%
\thanks{Department of Pure Mathematics and Mathematical Statistics,
Wilberforce Road, Cambridge CB3 0WB, UK}}

\begin{document}

\maketitle

\begin{abstract}
Let $q$ be an odd prime power and let $f(r)$ be the minimum size of the symmetric
difference of $r$ lines in the Desarguesian projective plane $PG(2,q)$.
We prove some results about the function $f(r)$, in particular showing that
there exists a constant $C>0$ such that $f(r)=O(q)$ for $Cq^{3/2}<r<q^2-Cq^{3/2}$.
\end{abstract}

\section{Introduction}

Let $q$ be an odd prime power and consider the Desarguesian projective plane
$PG(2,q)$.
(For detailed definitions of lines, coordinates, conics, etc, see, e.g., the monograph
Hirschfeld~\cite{Hirs}.)
Write $\cP$ and $\cL$ for the set of points and lines of $PG(2,q)$
respectively. We shall consider the subsets of $\cP$ or $\cL$ as elements of
a vector space isomorphic to $\F_2^N$, $N:=q^2+q+1$, and will switch between
the `subset' and `vector' interpretations without further comment. For
example, for subsets $A$ and $B$ of $\cP$ or $\cL$, $A+B$ represents the
symmetric difference of $A$ and~$B$.

Define for $0\le r\le N$,
\begin{equation}\label{f1}
 f(r)=\min\Big\{\big|\sum_{i=1}^r \ell_i\big|
 :\ell_1,\dots,\ell_r\in\cL\text{ distinct}\Big\},
\end{equation}
that is the minimal symmetric difference of $r$ lines in $PG(2,q)$.

The problem of determining $f(r)$ is motivated by the fact that it is an
algebraic
version of the Besicovitch-Kakeya~\cite{Besi} problem in a projective plane ---
determining the minimum size of a set that contains lines (or segments) in many directions.
For more results on Kakeya's problem in the finite fields see~\cite{Faber, BlokMazz}
 and the references there.

Given a set $R$ of lines in $PG(2,q)$, call a point {\em odd\/} if it is
incident with an odd number of lines in $R$, and define the terms
`even point', `single point', `double point', etc., analogously.
Let $\cP^o(R)$ be the set of odd points, and let $\cP^e(R)$, $\cP^k(R)$,
$\cP^{\ge k}(R)$ be defined analogously as the set of points that are even,
multiplicity $k$, and multiplicity at least~$k$, respectively.

Dually, for $S\subseteq\cP$, define $\cL^o(S)$ to be the set of lines
$\ell\in\cL$ such that $|\ell\cap S|$ is odd. Define $\cL^e(S)$, $\cL^k(S)$,
and $\cL^{\ge k}(S)$ analogously.

By duality of lines and points in the projective plane $PG(2,q)$ we can
rewrite \eqref{f1} in the equivalent forms
\begin{equation}\label{f2}
 f(r)=\min_{R\subseteq\cL,\,|R|=r}|\cP^o(R)|
 =\min_{S\subseteq\cP,\,|S|=r}|\cL^o(S)|.
\end{equation}
We shall therefore often switch the viewpoint and consider sets of points
which have odd intersections with few lines.

The next observation, proved below, is that $\cP^o(R)$ almost
determines~$R$, and $\cL^o(S)$ almost determines~$S$. Indeed, the
$N$ vectors specified by $\cL$ span an  $(N-1)$-dimensional subspace of
$\F_2^{\cP}$ and their only linear dependency is $\sum_{\ell\in\cL}\ell=0$.
This gives that $\cP^o(R)=\cP^o(R')$ iff either $R=R'$ or $R'=\cL\setminus R$.
Indeed, it is well known that the $N\times N$ point line 0--1 incidency
matrix $A$ has rank $N-1$ (one can consider $AA^T=J+qI$ and this has rank
$N-1$ over $\F_2$, see, e.g., Ryser~\cite{Ry}). The following useful lemma
is based on this observation.

\begin{lemma}\label{le:1}
 If\/ $R=\cL^o(S)$ then $|R|$ is even and either $S=\cP^e(R)$
 (if\/ $|S|$ is odd) or $S=\cP^o(R)$ (if\/ $|S|$ is even).
 Dually, if\/ $S=\cP^o(R)$ then $|S|$ is even and either $R=\cL^e(S)$
 (if\/ $|R|$ is odd) or $R=\cL^o(S)$ (if\/ $|R|$ is even).
\end{lemma}
\begin{proof}
The maps $\cL^o$ and $\cP^o$ can be thought of as $\F_2$-linear maps
between the set of subsets of $\cP$ and $\cL$, each regarded as a vector
space isomorphic to~$\F_2^N$. For $p\in\cP$,
$|\cL^o(\{p\})|=|\{\ell\in\cL:p\in\ell\}|=q+1$ is even, so $|\cL^o(S)|$
is even for all $S\subseteq\cP$. Moreover
\[
 \cP^o(\cL^o(\{p\}))=\sum_{\ell\ni p}\ell=\cP-\{p\}\in \F_2^{\cP}
\]
as the number $q+1$ of lines through $p$ is even and there is a unique
line through $p$ and $p'$ for every $p'\ne p$. By linearity,
$\cP^o(\cL^o(S))=\sum_{p\in S}(\cP-\{p\})=S$ when $|S|$ is even, and
so $\cP^o$ has rank at least $N-1$. Also, $\cP^o(\cL)=\emptyset$
as every point is in an even number of lines. Hence the kernel of
$\cP^o$ is $\{0,\cL\}$. Similarly the kernel of $\cL^o$ is $\{0,\cP\}$.
The result now follows as $\cP^e(R)=\cP\setminus \cP^o(R)$
and $\cL^e(R)=\cL\setminus\cL^o(R)$.
\end{proof}

\begin{lemma}\label{l:sym}
 For $0\le r\le N$, $f(N-r)=f(r)$.
\end{lemma}
\begin{proof}
Replacing any set $R=\{\ell_1,\dots,\ell_r\}$ by its complement
$\cL\setminus R$ and noting that
$\sum_{\ell\notin R}\ell=\sum_{\ell\in R}\ell$, we find that $f(N-r)\le f(r)$.
Reversing the roles of $r$ and $N-r$ gives $f(N-r)\ge f(r)$.
\end{proof}

\begin{lemma}\label{l:cong}
 Let\/ $R$ be any set of\/ $r$ lines in~$\cL$. Then
 \[
  r(q+2-r)\le|\cP^o(R)|\le rq+1
 \]
 and
 \[
  |\cP^o(R)|\equiv r(q+2-r)\bmod4.
 \]
 In particular, $f(r)\ge r(q+2-r)$ and\/ $f(r)\equiv r(q+2-r)\bmod 4$.
\end{lemma}
\begin{proof}
Each line of $R$ contains at least $q+1-(r-1)=q+2-r$ points that do not lie
on any other line of~$R$. Thus there are at least $r(q+2-r)$ points lying on
a single line, and so in particular $|\cP^o(R)|\ge r(q+2-r)$. On the other
hand, one line contains $q+1$ points and the symmetric difference of two
lines contains exactly $2q$ points. Thus $|\cP^o(R)|\le rq+1$ for $r\le 2$.
For $r>2$ write $R=R'\cup\{\ell,\ell'\}$. Then by induction
\begin{align*}
 |\cP^o(R)|&=|\cP^o(R')+\cP^o(\{\ell,\ell'\})|\\
 &\le |\cP^o(R')|+|\cP^o(\{\ell,\ell'\})|\\
 &\le ((r-2)q+1)+2q=rq+1.
\end{align*}
Now let $t_i=|\cP^i(R)|$ be the set of points of multiplicity~$i$. Then
$\sum it_i=r(q+1)$ is the number of points in all the lines counted with
multiplicity, and $\sum i(i-1)t_i=r(r-1)$ is the number of intersection
points between ordered pairs of lines counted with multiplicity.
Subtracting gives $\sum i(2-i)t_i=r(q+2-r)$. But $i(2-i)\equiv 0\bmod 4$
when $i$ is even and $i(2-i)\equiv1\bmod 4$ when $i$ is odd.
Thus $r(q+2-r)\equiv \sum_{i\text{ odd}}t_i=|\cP^o(R)|\bmod 4$.
\end{proof}

The function $f(r)$ is easily determined for $0\le r\le q+1$ (and hence
by \Lm{sym} also for $N-q-1\le r\le N$).

\begin{theorem}\label{t:init}
 For $0\le r\le q+1$, $f(r)=r(q+2-r)$.
\end{theorem}
\begin{proof}
\Lm{cong} implies $f(r)\ge r(q+2-r)$, so it remains by \eqref{f2} to
construct a set $S$ of points with $|S|=r$ and $|\cL^o(S)|=r(q+2-r)$.

Let $C=\{[s^2\cl st\cl t^2]:[s\cl t]\in PG(1,q)\}$ be the conic $XZ=Y^2$. We note
that all lines $\ell$ intersect $C$ in at most 2 points, and $|\ell\cap C|=1$
if and only if $\ell$ is one of the $q+1$ tangent lines to $C$.

Let $S$ be any subset of $C$ of size~$r$. No line intersects $S$ in more
than two points and so for any $p\in S$ exactly $r-1$ lines through $p$
meet $C$ at another point of~$S$, while $(q+1)-(r-1)=q+2-r$ lines through $p$
fail to meet $C$ at any other point of~$S$. Thus there are exactly $r(q+2-r)$
lines that meet $S$ in an odd number of points and so $|\cL^o(S)|=r(q+2-r)$
as required.
\end{proof}

The function $f(r)$ cannot vary too rapidly; trivially we have
$|f(r+1)-f(r)|\le q+1$. In fact, we can say slightly more.

\begin{theorem}\label{t:lip}
 For $0<r<N-2$, $|f(r+1)-f(r)|\le q-1$.
\end{theorem}
Note that $f(0)=f(N)=0$ and $f(1)=f(N-1)=q+1$, so this result fails for
$r=0,N-1$. On the other hand, the inequality can be sharp. For example,
$f(2)-f(1)=f(q+1)-f(q)=q-1$ by \Th{init}. There are other examples,
e.g., $f(2q-1)=q+1$ and $f(2q)=2$ (see \Th{2q} below).
\begin{proof}
Assume $|R|=r$ and $\cP^o(R)=S$ with $|S|=f(r)$.
Note that $S\ne\emptyset$ as $R\ne\emptyset,\cL$.
Pick $p\in S$. Assume every line $\ell$ through $p$ intersects $S$ in an odd number
of points. Then every line through $p$ intersects $S\setminus p$ is an even number
of points. Since distinct lines through $p$ partition $S\setminus p$,
we see that $|S\setminus p|$ is even and hence $|S|$ is odd,
contradicting Lemma~\ref{le:1}.    %%% it contradicts.
Thus there exists a line $\ell_e$ that meets $S$ in an even
(and positive) number of points. If all $\ell\in\cL$ met $S$ in an even number
of points then $\cL^o(S)=\emptyset$ and so $S=\emptyset$ or~$\cP$, a contradiction.
Thus there exists a line $\ell_o$ that meets $S$ in an odd number of points.
As $R=\cL^o(S)$ or $\cL^e(S)$, either $\ell_e$ or $\ell_o$ fails to lie in~$R$.
Adding such a line to $R$ increases $r$ by one and increases $S$ by at most $q-1$,
implying $f(r+1)-f(r)\le q-1$.

Replacing $r$ by $N-r-1$ and applying \Lm{sym}
gives $f(r+1)-f(r)=-(f(N-r)-f(N-r-1))\ge -(q-1)$, completing the
proof of \Th{lip}.
\end{proof}

\section{The case of $q+2$ lines}

Our next aim is to prove that the jump $f(q+2)-f(q+1)=f(q+2)-(q+1)$
is not too small.

\begin{theorem}\label{t:f(q+2)}\,
$f(q+2)= 2q-2$ for $q\le 13$.
More generally, for $q\ge 7$ we have $\frac32(q+1)\le f(q+2)\le 2q-2$.
\end{theorem}

To prove this we shall use several lemmas, some classical results of this
 topic. Most of their proofs use either R\'edei's method
 (see. e.g.,~\cite{LoSc}) or some version of Combinatorial Nullstellensatz
 (see, e.g.,~\cite[Theorem~1.2]{CN}).
Arrangements of $q+2$ lines are the most investigated part of finite geometries.
In the following, a {\em triple point\/} with respect to a set
of lines $R$ will refer to a point which lies on {\em at least\/} three
lines.

\begin{lemma}[Bichara and Korchm\'aros~\cite{BichKorc}]\label{l:first}\,
 Let\/ $R$ be a set of\/ $q+2$ lines in $PG(2,q)$.
Then there are at most two lines without triple points.  \qed
\end{lemma}

A {\em blocking set\/} in the affine plane $AG(2,q)$ or %and
in the projective plane $PG(2,q)$ is a set $B$ of points such
that each line is incident with at least one point of~$B$.

\begin{lemma}[Brouwer and Schrijver~\cite{BrSc} and Jamison~\cite{Jami}]\label{l:second}\,
 Let\/ $B$ be a blocking set in $AG(2,q)$.
 Then $B$ consists of at least\/ $2q-1$ points. \qed
\end{lemma}

\begin{lemma}[Sz\H{o}nyi~\cite{Szonyi}]\label{l:szonyi}\,
 Let\/ $B$ be a {\rm minimal} blocking set in $PG(2,q)$
 of size less than $3(q+1)/2$  where $q=p^h$ for some prime~$p$. % p
 Then all lines meet $B$ in $1$ {\rm mod} $p$ points. \qed
\end{lemma}

The following lemma is contained in~\cite{BlokMazz} (top of page~211) %-the
 as a part of a more complex argument.
For completeness we reproduce its proof here.

\begin{lemma}[Blokhuis and Mazzocca~\cite{BlokMazz}]\label{l:third}\,
 Let\/ $R$ be a set of\/ $q+2$ lines with at least one of the lines
 containing no triple points. Then the number of odd points is at least $2q$
 minus the number of lines in $R$ without triple points.
\end{lemma}
\begin{proof}
Without loss of generality, we may assume that $R$ contains the line at
infinity and that this line has no triple point. Let $L$ be the set of $q+1$
lines in $AG(2,q)$ obtained by restricting the remaining lines of $R$ to
$AG(2,q)$. As the line at infinity contains no triple point, no two lines
in $L$ are parallel. Then as $|L|=q+1$, every line $\ell$ in $AG(2,q)$ is
parallel to precisely one line of~$L$.

\noindent{\bf Claim.}
In $AG(2,q)$ the odd points block all lines in $AG(2,q)$,
except those in $L$ that have no triple points.

Indeed, assume first that $\ell \notin L$. Then $\ell$ intersects $q$ of
the lines in~$L$; indeed it intersects all but the unique line in $L$
parallel to~$\ell$. Since $q$ is odd, $\ell$ has an odd point.

Now assume $\ell\in L$ and has a triple point. As there are $q$
points in $L$ and only $q$ other lines in~$L$, the fact that some
point in $\ell$ meets at least two of these lines implies that there
is a point of $\ell$ which meets no other line of~$L$. Such a point
is a single (and hence odd) point.

Adding one point from each line without a triple point (except the
line at infinity) we obtain a blocking set of the affine plane, which
by \Lm{second} contains at least $2q-1$ points. The result
follows.
\end{proof}

\begin{proof}[Proof of the lower bound in \Th{f(q+2)}]\,
Let $R$ be a set of $q+2$ lines with $f(q+2)=|\cP^o(R)|$, $S:=\cP^o(R)$,
and let $T_3$ be the set of triple points. % and let ... be
We will show  that $|S|\ge 3(q+1)/2$.

First, suppose that $R$ has a line without a triple point.
Then by Lemmas~\ref{l:first} and \ref{l:third} there
are at least $2q-2$ odd points.

Second, suppose all $q+2$ lines in $R$ have triple points and $|S|<2q-2$.
Since  $f(q+2)\equiv 0 \bmod4$ by Lemma~\ref{l:cong}
we may suppose that $|S|\le 2q-6$.

\noindent{\bf Claim.} $S$ is a  minimal blocking set in $PG(2,q)$.

Indeed, every line $\ell$ in $PG(2,q)$ is either in our set
(in which case it contains a single point), or intersects all $q+2$ lines of~$R$.
As $q+2$ is odd, $\ell$ must contain an odd point.

That $S$ is minimal can be seen as follows:
Let $v\in S$ and suppose on the contrary that $S\setminus\{v\}$ meets all lines.
Since $v$ is an odd point, there are $2m+1$ lines of $R$ containing it.
Each of these lines contains at least $2m-1$ additional odd (single) points of~$S$.
Moreover, every line $\ell$ not in $R$ has an odd number of odd points.
Then if $\ell\notin R$ is a line through $v$, we have
$|S\cap \ell|\ge 2$ and hence $|S\cap \ell|\ge 3$. % expanded
In total we find at least $(2m+1)(2m-1)+ 2(q-2m)\geq 2q-1$ odd points beside $v$.
This contradiction completes the proof of the Claim.

We count multiplicities of intersections
as in the proof of \Lm{cong}. If we let $t_i$ be the number of points
that occur in exactly $i$ of our lines, then
$\sum_i it_i=\sum_i i(i-1)t_i=(q+2)(q+1)$. Thus $\sum_i i(i-2)t_i=0$,
 rearranging
\begin{equation}\label{eq:3}
  |S|=\sum_{i\text{ odd}}t_i
  =\sum_{i \geq 3} \left( i(i-2)+ (i \bmod 2)\right)t_i
  =4t_3+ 8t_4+ 16t_5+24t_6+\dots
\end{equation}

Let $R_3\subseteq R$ be the set of lines having a single triple point,
 and that point has degree three, and
 let $R_4\subseteq R$ be the set of lines having a single triple point,
 and that point has degree at least four.
Every line in $R$ has at least one triple point,
 the members of $R\setminus (R_3\cup R_4)$ have at least two.
So adding up the degrees of triple points we obtain
 $\sum_{i\ge 3} it_i = \sum_{\ell\in R} |\ell\cap T_3|\geq 2|R|-|R_3|-|R_4|$.
Consider  $\sum_{i\ge 4}it_i$, it is an upper bound for $|R_4|$.
Summarizing  we obtain
\[
    3t_3+ \sum_{i\ge 4}2it_i\ge 2|R|-|R_3|.
\]
This and \eqref{eq:3} yield $|S|\geq 2q+4-|R_3|$.
Every $R_3$ line meets $S$ in two elements, so %-a
 actually $R_3=\emptyset$ by Lemma~\ref{l:szonyi}
 for $|S| < 3(q+1)/2$.
This contradiction completes the proof of $|S|\ge 3(q+1)/2$.
For $q\le 13$ we note that $3(q+1)/2>2q-6$, so $f(q+2)=2q-2$. % added
\end{proof}

Finally, to show $f(q+2)\le 2q-2$ recall that $f(q+2)\le f(q+1)+(q-1)=2q$
by Theorems \ref{t:lip} and \ref{t:init}, while $f(q+2)\equiv 0\bmod 4$
by \Lm{cong}. Thus $f(q+2)\le 2q-2$.

This upper bound on $f(q+2)$ can also be seen in the following way.
There is an action of $SL(2,q)$ on $PG(2,q)$ in which the orbits
are $A$, $B$, and $C$, where $C$ is the conic described above,
$A$ is the set of points which lie on no tangent of $C$ and $B$ is the set
of points that lie on two tangents of $C$. Now $|\cL^o(C)|=q+1$, so if % A<->B
$p\in A$ then $|\cL^o(C\cup\{p\})|=(q+1)+(q+1)$ as all lines through $p$
change from having an even intersection with $C$ to having an odd intersection
with $C\cup\{p\}$. On the other hand, if $p\in B$ then
$|\cL^o(C\cup\{p\})|=(q+1)+(q-1)-2=2q-2$ as there are $q-1$ lines thorough
$p$ with an even intersection with $C$ and an odd intersection with $C\cup\{p\}$,
while there are 2 lines through $p$ that are tangent to $C$ and so have
odd intersection with $C$ and even intersection with $C\cup\{p\}$.
The result now follows from~\eqref{f2}.

We conjecture that in fact the upper bound is correct in \Th{f(q+2)}.

\begin{conjecture}
$f(q+2)=2q-2$.
\end{conjecture}

\section{Exact values near $2q$}

A few more values of $f(r)$ are known when $r$ is small. To derive % more
these we shall make use of the following result.

\begin{lemma}\label{l:rs}
 For even $s$, $f(s)$ is the minimum even $r$ such that there exists a % even
 set\/ $R$ of lines with $|R|=r$ and $|\cP^o(R)|=s$.
\end{lemma}
\begin{proof}
Assume $R$ is a set of lines with $|R|=r$ and $\sum_{\ell\in R}\ell=S$
with $|S|=s$. Now $|\cL^o(S)|$ is even while $|\cL^e(S)|$ is odd. Hence
$R=\cL^o(S)$ as $r$ is even. Thus, by \eqref{f2}, $f(s)\le r$. Conversely,
if $f(s)=r$ and $|S|=s$ with $|\cL^o(S)|=r$, then $r$ is even and, % r is even
setting $R=\cL^o(S)$, we have $|R|=r$ and $|\cP^o(R)|=|S|=s$ as $s$ is even.
\end{proof}

\begin{theorem}\label{t:2q}
 $f(2q-1)=q+1$, $f(2q)=2$, $f(2q+1)=q-1$.
\end{theorem}
\begin{proof}
If $|R|=2$ then $|\cP^o(R)|=2q$, so $f(2q)\le2$ by \Lm{rs}.
However $f(r)>0$ and $f(r)$ is even for $0<r<N$, so $f(2q)=2$.
Thus $f(2q-1),f(2q+1)\le q+1$ by \Th{lip}. Also
$f(2q+1)\equiv(2q+1)(-q+1)\equiv q-1\bmod 4$
and $f(2q-1)\equiv (2q-1)(-q+3)\equiv q+1\bmod 4$
by \Lm{cong}. Thus it is sufficient to show that
$f(2q\pm1)>q-3$. As $2q\pm1$ is odd, there exists a $R$ with
$|R|=f(2q\pm1)$ and $|\cP^o(R)|=N-(2q\pm1)\ge q^2-q$.
But $|\cP^o(R)|\le q|R|+1$ by \Lm{cong}, so $|R|>q-3$.
\end{proof}

\section{A graph clique decomposition lemma}

The values of $f(r)$ for $q+2<r<2q-1$ remain to be determined,
and indeed $f(r)$ is unknown for many values of $r<Cq^{3/2}$, although
some non-trivial bounds are given by Lemmas \ref{l:32} and \ref{l:e}
below. For larger $r$, between $Cq^{3/2}$ and $N-Cq^{3/2}$, we shall
show much more. Indeed it seems that $f(r)$ can be determined
for most values of $r$ in this range, although an explicit
description of these values seems difficult.

Suppose that $s$ is even (the case when $s$ is odd follows by
considering $f(N-s)$). By \Lm{rs} and duality it is enough to determine
for each even $r$ in turn whether or not there exists a set $S$
of points such that $|\cL^o(S)|=s$. Any set of points $S$ induces
an edge-decomposition of the complete graph $K_S$ with vertex set $S$
into cliques on the sets $\ell\cap \cal S$, $\ell\in\cL$. Indeed,
every pair of points of $S$ lie in a unique line $\ell\in\cL$
so each edge $K_S$ lies in a unique clique $K_{\ell\cap S}$.
We show that $s=|\cL^o(S)|$ can be determined in terms of
the sizes of these cliques.

\begin{lemma}\label{l:cv}
 Suppose $r=|S|$ is even and $|\cL^o(S)|=rq-4t$. For $\ell\in\cL$
 write $r_\ell=|S\cap\ell|$. Then
 $\sum_{\ell\in\cL}\big\lfloor\frac{r_\ell}{2}\big\rfloor=\frac{r}{2}+2t$.
\end{lemma}
\begin{proof}
As there are $q+1$ lines through each point of~$S$,
$\sum_{\ell\in\cL} r_\ell=r(q+1)$. Thus
\[
 rq-4t=|\cL^o(S)|=\sum_{r_\ell\text{ odd}}1
 =\sum_\ell (r_\ell-2\big\lfloor\tfrac{r_\ell}{2}\big\rfloor)
 =rq+r-2 \sum_\ell \big\lfloor\tfrac{r_\ell}{2}\big\rfloor.
\]
Hence $\sum\big\lfloor\frac{r_\ell}{2}\big\rfloor=\frac{r}{2}+2t$.
\end{proof}

Note that by \Lm{cong} $s=|\cL^o(S)|$ must be of the form $rq-4t$
with $0\le t\le \binom{r}{2}$. Since we are interested in the smallest $r$
for which a suitable set $S$ exists, typically we expect $t$ to
be relatively small and $r$ not much bigger that $s/q$. We can therefore
reduce the problem to the question of (a) whether there is {\em any}
clique decomposition of $K_r$ into cliques of size $r_1,\dots,r_n$
with a given value of $\sum\big\lfloor\frac{r_i}{2}\big\rfloor$,
and (b) whether such a decomposition can be realized by a set
of points inside $PG(2,q)$.

We call an edge-decomposition $\Pi$ of $K_r$ into cliques
of orders $r_1,\dots,r_n$ a {\em simple decomposition} if there is at
most one value of $i$ with $r_i>3$. In other words, $K_r$ is decomposed
as single edges, triangles, and at most one larger clique.
We write $M(\Pi)$ for the sum $\sum_{i=1}^n\big\lfloor\frac{r_i}{2}\big\rfloor$.

\begin{lemma}\label{l:krd}
 Suppose we are given an edge-decomposition $\Pi$ of\/ $K_r$
 with $M(\Pi)<\frac14 r(\sqrt{4r-3}-1)$. Then there exists
 a simple edge-decomposition $\Pi'$ of\/ $K_r$ with $M(\Pi')=M(\Pi)$.
\end{lemma}
\begin{proof}
Assume $\Pi$ decomposes $K_r$ into cliques of orders $r_1,\dots,r_n$
with $r_1\ge r_2\ge\dots\ge r_n$.
Let $C_i$ be the $i$'th clique. Then there are
$r_1(r-r_1)$ edges from $V(C_1)$ to $V(K_r)\setminus V(C_1)$. Moreover,
each clique $C_i$, $i>1$, can meet $C_1$ in at most
one vertex and hence covers at most $r_i-1$ of these edges.
Thus $\sum_{i>1}(r_i-1)\ge r_1(r-r_1)$ and hence
\begin{equation}\label{e:b1}
 M(\Pi)\ge \sum_{i=1}^n\frac{r_i-1}{2}\ge
 \frac{r_1-1}{2}+\frac{r_1(r-r_1)}{2}.
\end{equation}
On the other hand there are $\binom{r}{2}$ edges to be covered
in total, so
\begin{equation}\label{e:b2}
 M(\Pi)\ge \sum_{i=1}^n\frac{r_i-1}{2}=\sum_{i=1}^n\frac{1}{r_i}\binom{r_i}{2} %=
 \ge \frac{1}{r_1}\binom{r}{2}.
\end{equation}
For $r_1<r/2$, the bound in  \eqref{e:b1} is increasing and the bound
in \eqref{e:b2} is decreasing as $r_1$ increases, so the smallest bound
on $M(\Pi)$ occurs when
the two bounds are equal. It can be checked that this
occurs when $r=r_1^2-r_1+1$ with a common bound
$M(\Pi)\ge \frac12r(r_1-1)=\frac14 r(\sqrt{4r-3}-1)$. This contradicts
the assumption on $M(\Pi)$, so we may assume $r_1\ge r/2$.

Let $E_1$ be the set of $r_1(r-r_1)$ edges joining $C_1$
to the rest of $K_r$ and $E_2$ be  the set of $\binom{r-r_1}{2}$
edges of $K_r$ not meeting~$C_1$. For each clique $C_i$, $i>1$, we note
that for all $r_i\ge2$,
\[
 |E_1\cap E(C_i)|-|E_2\cap E(C_i)|\le \Big\lfloor\frac{r_i}{2}\Big\rfloor
 \le |E_1\cap E(C_i)|+|E_2\cap E(C_i)|.
\]
Indeed, the right hand side is just $\binom{r_i}{2}$, while the left
hand side is either $(r_i-1)-\binom{r_i-1}{2}$
or $-\binom{r_i}{2}$ depending on whether or not $C_i$ meets some
vertex of~$C_1$. Note that the lower bound is achieved if $r_i\in\{2,3\}$
and $C_i$ meets~$C_1$. Summing over all cliques gives
\begin{equation}\label{e:b3}
 \Big\lfloor\frac{r_1}{2}\Big\rfloor+|E_1|-|E_2|\le
 M(\Pi)\le \Big\lfloor\frac{r_1}{2}\Big\rfloor+|E_1|+|E_2|.
\end{equation}
Also note that
$\lfloor\frac{r_i}{2}\rfloor\equiv\binom{r_i}{2}\bmod 2$, so that
$M(\Pi)$ is equivalent to either bound modulo~2.

As $r_1\ge r/2$, the graph on $E_1\cup E_2$ can be packed with
$|E_2|$ triangles each meeting $C_1$. Indeed, it is enough to decompose
$K_{r-r_1}$ completely into at most $r_1$ partial matchings
$M_1,\dots, M_{r_1}$ and then join each matching to a distinct
vertex of $C_1$ to obtain sets of edge-disjoint triangles.
For even $r-r_1$, it is well-known that $K_{r-r_1}$
can be decomposed into $r-r_1-1<r_1$ perfect matchings.
For odd $r-r_1$ decompose $K_{r-r_1+1}$ into $r-r_1\le r_1$
perfect matchings and remove a single vertex to give a decomposition
of $K_{r-r_1}$ into $r-r_1$ partial matchings. Completing the
packing of $E_1\cup E_2$ by including $K_2$s covering the
remaining edges of $E_1$ gives a decomposition $\Pi''$ of $K_r$ which achieves
the lower bound $M_0=\lfloor r_1/2\rfloor+|E_1|-|E_2|$ in \eqref{e:b3}.
Now replacing $(M(\Pi)-M_0)/2\le |E_2|$ of the triangles of this packing
with three $K_2$s, allows us to increase $M(\Pi'')$ in steps of 2 until
we get to a packing $\Pi'$ of $C_1$, edges, and triangles, with
$M(\Pi')=M(\Pi)$. % added Pi''
\end{proof}

\begin{lemma}\label{l:krd2}
 Let\/ $m=\lceil\sqrt{r-3}\rceil-1$. Then for any integer $s$
 with $s\le\binom{r}{2}$, $s\equiv\binom{r}{2}\bmod2$, and
 $s\ge \lfloor\frac{r-m}{2}\rfloor+\frac{m}{2}(2r-3m+1)$ there exists
 a simple decomposition $\Pi$ of $K_r$ with $M(\Pi)=s$.
\end{lemma}
\begin{proof}
From the proof of \Lm{krd} we know that we can construct a simple
a decomposition for any $s\equiv \binom{r}{2}$ and
\[
 \big\lfloor\tfrac{r_1}{2}\big\rfloor+r_1(r-r_1)-\tbinom{r-r_1}{2}
 \le s\le \big\lfloor\tfrac{r_1}{2}\big\rfloor+r_1(r-r_1)+\tbinom{r-r_1}{2}
\]
with $r_1\ge\frac{r}{2}$.
It is a simple but tedious exercise to show that the intervals
for $r_1=\lceil\frac{r}{2}\rceil,\dots,r-m$
cover every $s\equiv\binom{r}{2}$
in the range from $\lfloor\frac{r-m}{2}\rfloor+\frac{m}{2}(2r-3m+1)$
to $\frac34\binom{r}{2}$. For $s>\frac34\binom{r}{2}$ it is enough
to show that one can pack $(\binom{r}{2}-s)/2\le\binom{\lfloor r/2\rfloor}{2}$
triangles into $K_r$. This also follows from the proof of \Lm{krd}
where it was shown that one can pack $\binom{\lfloor r/2\rfloor}{2}$
triangles into $K_r\setminus E(K_{\lceil r/2\rceil})$.
\end{proof}

Lemmas \ref{l:krd} and \ref{l:krd2} show that if there exists a decomposition
with $M(\Pi)=s$ then there exists a simple decomposition with $M(\Pi)=s$
except possibly in the range between about $\frac{1}{2}r^{3/2}$ and
about $r^{3/2}$. There can exist non-simple decompositions in this range
for which there is no simple decomposition. For example, the lines
of a projective plane of order $q'$, $q'$ odd, give rise to a decomposition $\Pi$ of $K_r$
when $r=q'^2+q'+1$ with $M(\Pi)=(q'^2+q'+1)(q'+1)/2$ (exactly the bound in \Lm{krd}).
One can check that for a simple decomposition to have the same
value of $M(\Pi)$ would require $\frac{q'-1}{2}<r_1<\frac{q'+1}{2}$
for large $q'$, an impossibility, so no corresponding simple decomposition exists.

\section{Realizing clique decompositions of the projective plane}

We now turn to the question of whether a simple decomposition can be realized
by a set of points in $PG(2,q)$. One needs a set $S$ formed by taking a large
number $r_1$ of points in one line, and the remaining points only on lines
intersecting $S$ in at most 3 points.
The proof of the following lemma provides a construction which
realizes this in most relevant cases.

\begin{lemma}\label{l:range}
 Fix $r$, $0\le r\le q+1$ and assume
 $r_1\ge\max\{\frac13(2r-3),(2r-3)-(q+1)\}$.
 Then any simple decomposition $\Pi$ of $K_r$ with maximal clique of
 order $r_1$ can be realized by a set of points in $PG(2,q)$.
\end{lemma}
\begin{proof}
Consider sets of points that are subsets of $C\cup L$, where $C=\{XZ=Y^2\}$ is
the conic used in the proof of \Th{init} and $L=\{X=dZ\}$ is a line that does
not intersect~$C$ (so $d$ is chosen to be a quadratic non-residue in the
field~$\F_q$). A simple calculation shows that the secant line joining
$[s^2\cl st\cl t^2]$ and $[s'^2\cl s't'\cl t'^2]$ on $C$ meets $L$ at the
point $[d(st'+s't)\cl dtt'+ss'\cl st'+s't]$ on~$L$. This mapping of pairs of
points on $C$ to $L$ is more easily described by introducing the norm group
$G=\F_{q^2}^\times/\F_q^\times$. The points $p=[s^2\cl st\cl t^2]\in C$
correspond to the coset $\phi(p)=(s+t\sqrt{d})\F_q^\times$ and the coset
$\alpha=(a+b\sqrt{d})\F_q^\times$ corresponds to the point
$\psi(\alpha)=[db\cl a\cl b]\in L$. The secant line through $p,p'\in C$ then
meets $L$ at $\psi(\phi(p)\phi(p'))$. The key point is that $G$ is cyclic of
order $q+1$. Hence by taking a subset $P=\{p_1,p_2,\dots,p_s\}$ of $C$ with
$2s-3\le q+1$ such that $\phi(p_i)$ form a suitable geometric progression, the secants % added details
through these points meet $L$ in only $2s-3$ points (assuming $s\ge2$).
Indeed, we can take $\phi(p_i)=\alpha^i$ where $\alpha$ is a generator of~$G$
so that the secants meet $L$ at the points
$\psi(\alpha^3),\psi(\alpha^4),\dots,\psi(\alpha^{2s-1})$.
Moreover there are 4 points ($\psi(\alpha^3),\psi(\alpha^4),\psi(\alpha^{2s-2}),\psi(\alpha^{2s-1})$)
on $L$ which meet just one secant, 4 which meet exactly 2 secants, etc., with
1 or 3 points meeting $\lfloor s/2\rfloor$ secants (depending on the parity
of~$s$).  Now let $P'=\{p'_1,\dots,p'_t\}$ be a set of $t$ points on the
line $L$ and suppose there are $k$ secants through two points of $P$ meeting $P'$.
then $P\cup P'$ induces a simple edge decomposition of $K_{P\cup P'}$ with
one clique of order $|P'|$ and $k$ triangles, the remaining cliques being %P'
single edges.

We now consider the conditions on the parameter that allow us to vary $k$
between the minimum of zero and the maximum of $\binom{s}{2}$, where $s\ge2$.
To achieve $k=0$ requires $t\le (q+1)-(2s-3)$ as $P'$ must avoid all the secant
lines through~$P$. To achieve $k=\binom{s}{2}$
requires $t\ge 2s-3$ as $P'$ must meet all secants through~$P$.
All values of $k$ between the minimum and maximum can
be achieved one step at a time by moving some point of $P'$ so that it
meets one more secant line. Now $s=r-r_1$ and $t=r_1$ so these conditions
become
\[
 r_1\le q+1-(2r-2r_1-3)\qquad\text{and}\qquad r_1\ge 2r-2r_1-3,
\]
or equivalently $r_1\ge (2r-3)-(q+1)$ and $r_1\ge\frac13(2r-3)$.
For $s<2$ there are no secant lines and the only restriction is
$t=r_1\le q+1$ which follows from $r_1\le r\le q+1$.
\end{proof}

\begin{corollary}
 There exists an absolute constant $C>0$ such that
 $w/q\le f(w)\le w/q+C(w^{3/2}/q^{5/2}+1)$ for all even $w$
 with $Cq^{3/2}\le w\le N-Cq^{3/2}$.
\end{corollary}
Note that for odd $w$, $N-w$ is even and so
$(N-w)/q\le f(w)=f(N-w)\le (N-w)/q+C((N-w)^{3/2}/q^{5/2}+1)$.
\begin{proof}
By choosing $C$ sufficiently large we may assume that $q$ is also large.
The lower bound follows from Lemmas \ref{l:rs} and \ref{l:cong}.
For the upper bound choose $r$ minimal such that
$r>w/q+2w^{3/2}/q^{5/2}$ and $r\equiv qw\bmod 4$.
Write $w=rq-4t$, so that $r^{3/2}\le 4t\ll r^2$ and $r>\sqrt{q}$.
By \Lm{krd2} there exists a simple decomposition of $K_r$
with $M(\Pi)=r/2+2t$ and indeed, this decomposition must
have maximal clique size $r_1=r-O(\sqrt{r})$.
Then by \Lm{range} this decomposition can be realised by a
subset $S$ of $PG(2,q)$. Now $|\cL^o(S)|=qr-4t=w$ by \Lm{cv}
and so $f(w)\le r\le w/q+C(w^{3/2}/q^{5/2}+1)$.
\end{proof}

\section{Further constructions from blocking sets
   and the maximum of $f(r)$}

We shall now provide some constructions that give
at least some reasonable bounds on $f(r)$ for $r<Cq^{3/2}$
or $r>N-Cq^{3/2}$.

Let $Q^+\subseteq\F_q$ be the set of non-zero quadratic residues and
$Q^-\subseteq\F_q$ be the set of quadratic non-residues.
Both sets have $(q-1)/2$ elements.
Define $Q_i\subseteq\cP$, $i=0,1$ by
\[
 Q_0=\{[x\cl0\cl1]:x\in Q^+\}\cup\{[1\cl x\cl 0]:x\in Q^+\}\cup\{[0\cl1\cl x]:x\in Q^-\},
\]
and
\[
 Q_1=\{[x\cl0\cl1]:x\in Q^+\}\cup\{[1\cl x\cl0]:x\in Q^+\}\cup\{[0\cl1\cl x]:x\in Q^+\}.
\]
Given any line $\ell\colon \alpha X+\beta Y+\gamma Z=0$ that does not go through
the points $O_x:=[1\cl0\cl0]$, $O_y:=[0\cl1\cl0]$, $O_z:=[0\cl0\cl1]$,
we have $|\ell\cap Q_i|\equiv i\bmod 2$.
Indeed, $\ell$ intersects $\{[x\cl0\cl1]:x\in Q^+\}$ iff $\alpha/\gamma\in Q^+$
and similarly for the others. But for any $\alpha,\beta,\gamma\ne0$
an odd number of the conditions $\alpha/\gamma\in Q^+$, $\beta/\gamma\in Q^+$,
and $\gamma/\alpha\in Q^+$ hold.

The example $Q_0$ is due to J. di Paola.
% It is also a blocking set, and
By a famous result of Blokhuis~\cite{Blok} the set
$Q_0\cup \{ O_x, O_y, O_z\}$ is the smallest nontrivial blocking set on $PG(2,q)$ when $q$ is prime.

\begin{lemma}\label{l:32}
\[
 f(\tfrac{3}{2}(q-1)+kq+j)\le 3q+j(q+2-j)
\]
for $0\le k\le(q-1)/2$ and $0\le j\le q+1$.
\end{lemma}
\begin{proof}
Let $V$ be the set of $kq$ points that lie in one of $k$ ``vertical'' lines of
the form $X=\alpha Z$, $\alpha\in Q^-$, not including the point $O_y$ at infinity.
Let $C$ be any set of $j$ points on the conic $XZ=Y^2$. Note that $V$, $Q_i$, and $C$
are pairwise disjoint for $i=0,1$. Let $S=V\cup Q_{k\bmod 2}\cup C$ so that
$|S|=\tfrac{3}{2}(q-1)+kq+j$. Consider a line $\ell$ that does
not meet $\{O_x,O_y,O_z\}$. Then $|\ell\cap V|=k$ and
$|\ell\cap Q_{k\bmod2}|\equiv k\bmod2$. Thus
$|\ell\cap S|\equiv |\ell\cap C|\bmod 2$. From the proof of \Th{init}
there are %%% at most
 $j(q+2-j)$ lines that meet $C$ in an odd number of points,
and there are only $3q$ lines that meet $\{O_x,O_y,O_z\}$, so
$f(|S|)\le |\cL^o(S)|\le 3q+j(q+2-j)$ as required.
\end{proof}

\begin{lemma}\label{l:e}
\[
 f(kq+j)\le k+j(q+2-j)
\]
for $0\le k\le(q-1)/2$, $k$ even, and $0\le j\le q+1$.
\end{lemma}
\begin{proof}
Let $V$ and $C$ be as in the proof of Lemma~\ref{l:32}. Then
the number of lines meeting $C$ in an odd number of points is
$j(q+2-j)$ while the number of lines meeting $V$ in an odd
number of points is just $k$ (the lines of $V$). As
$|V\cup C|=kq+j$, $f(kq+j)\le k+j(q+2-j)$.
\end{proof}

\begin{lemma}\label{l:o}
\[
 f(q+1+kq+j)\le q+1+k+j(q+2-j)
\]
for $0\le k\le(q-1)/2$, $k$ even, and $0\le j\le q-1$,
\end{lemma}
\begin{proof}
Let $V$ and $C$ be as in the proof of Lemma~\ref{l:32} except that we
shall now insist that $O_x,O_z\notin C$. Let $C'$ be the conic
$XZ=4Y^2$. Note that $C'$ could only meet $C$ at the points $O_x,O_z$,
which we have assumed do not lie in $C$. Also $C'\cap V=\emptyset$.
There are $q+1$ lines that meet $C'$ in an odd number of points,
$j(q+2-j)$ lines that meet $C$ in an odd number of points, and
$k$ lines that meet $V$ in an odd number of points. The result
follows since $|V\cup C\cup C'|=q+1+kq+j$.
\end{proof}

\begin{corollary}
 For large $q$,
 the maximum value of $f(r)$ is $(q^2+4q+3)/4$ and occurs only
 at $r=(q+1)/2$, $r=(q+3)/2$, $r=N-(q+1)/2$, and $r=N-(q+3)/2$.
\end{corollary}
\begin{proof}
The result follows when $r$ is restricted to the range
$0\le r\le q+1$ and $N-(q+1)\le r\le N$ by \Th{init}
and \Lm{sym}, so it is enough by \Lm{sym} to bound
$f(r)$ in the range $r\in[q+2,N/2]$.
For $r\in[q+2,(\frac32-\varepsilon)q]$ we can apply Lemma~\ref{l:o}
with $k=0$ to obtain $f(r)\le (\frac14-\varepsilon^2)q^2+O(q)$.
For $r\in[(\frac32-\varepsilon)q,\frac32(q-1)]$
we can apply Lemma~\ref{l:32} with $k=j=0$ and Theorem~\ref{t:lip}
to obtain $f(r)\le 3q+(q-1)\varepsilon q$. Thus we may assume
$r\ge\frac32(q-1)$.

If $|r/q-t|\ge\frac{1}{4}$
for every integer $t$, then we write $r=\frac{3}{2}(q-1)+kq+j$,
where either $0\le j\le\frac{3}{2}+\frac{q}{4}$
or $\frac{3}{2}+\frac{3q}{4}\le j<q$. In either case
Lemma~\ref{l:32} implies
\[
 f(r)\le 3q+\tfrac{q+5}{4}\cdot\tfrac{3q+3}{4}=\tfrac{1}{16}(3q^2+66q+15).
\]
If $|r/q-t|<\frac{1}{4}$ and $\lfloor (r-1)/q\rfloor$ is even, we
write $r=kq+j$ with $1\le j<\frac{q}{4}$ or
$\frac{3q}{4}<j\le q$. In either case Lemma~\ref{l:e} gives
\[
 f(r)\le k+\tfrac{3q+1}{4}\cdot\tfrac{q+7}{4}\le \tfrac{1}{16}(3q^2+30q-1).
\]
Finally, if $|r/q-t|<\frac{1}{4}$ and $\lfloor (r-1)/q\rfloor$ is odd,
we write $r=q+1+kq+j$ with $0\le j<\frac{q}{4}-1$ or
$\frac{3q}{4}-1<j\le q$. In either case Lemma~\ref{l:o} gives
\[
 f(r)\le q+1+k+\tfrac{3q-3}{4}\cdot\tfrac{q+11}{4}\le \tfrac{1}{16}(3q^2+38q+24).
\]
Thus in all cases
\[
 f(r)\le \tfrac{1}{16}(3q^2+66q+15)<\tfrac{1}{4}(q^2+4q+3).
\]
for $q$ sufficiently large.
\end{proof}

\section{Exact values from the Baer subplane}

A subset of points $S\subseteq\cP$ is % generating
 a {\em subplane of order} $k$
if $|S|=k^2+k+1$ and the sets $\{\ell\cap S:\ell\in\cL,\,|\ell\cap S|>1\}$ form
the line system of a finite projective plane of order~$k$.
In the case when $k=\sqrt{q}$, we call $S$ a {\em Baer subplane}. It is well known that
such Baer subplanes exists whenever $q$ is a perfect square (see Bruck~\cite{Bruck}).
Even more (see, e.g., Yff~\cite{Yff}) $\cP$ can be partitioned into
$q-\sqrt{q}+1$ Baer subplanes.

Consider a Baer sublane $B$ and let $R_B\subseteq\cL$ be the set of lines meeting
it in exactly $\sqrt{q}+1$ points. Then $|R_B|=q+\sqrt{q}+1$.
The lines of $R_B$ cover every point of $B$ exactly $\sqrt{q}+1$ times, and
every other point exactly once. Thus $\cP^o(R_B)=\cP\setminus B$, which is very
large. However, consider an arbitrary point $p\notin B$ and
let $R$ be the symmetric difference of $R_B$ and $\cL(\{p\})$
(these two families contain only one common line $\ell_p\in R_B$ through~$p$).
Then $\cP^o(R)=B\cup\{p\}$.  We obtain
\begin{equation}\label{eq:7}
 f(2q+\sqrt{q})\le q+\sqrt{q}+2.
\end{equation}
Considering $p\in B$ and the set of even lines of $B\setminus\{p\}$
(it is again the symmetric difference of $R_B$ and $\cL(\{p\})$, now they
have $\sqrt{q}+1$ common lines) we obtain
\begin{equation}\label{eq:8}
 f(2q-\sqrt{q})\le q+\sqrt{q}.
\end{equation}
Considering two disjoint Baer subplanes we get
\begin{equation}\label{eq:9}
 f(2q+2\sqrt{q}+2)\le 2q+2\sqrt{q}+2.
\end{equation}

\begin{theorem}\label{t:e78}
 Equality holds in \eqref{eq:7} and \eqref{eq:8} for $q\ge81$.
\end{theorem}
We also {\bf conjecture} that equality holds in \eqref{eq:9}, too
(at least for large enough~$q$). For the proof of \Th{e78} we need the following
classical results and a few lemmata.

\begin{lemma}\label{l:Bruen}
 {\rm (Bruen~\cite{Bruen}, sharpening by Bruen and Thas~\cite{BruenThas})}\\
 Suppose that $S\subseteq\cP$ is a nontrivial blocking set (i.e., it meets every
 line but does not contain any) then $|S|\ge q+\sqrt{q}+1$.
 Moreover, if\/ $|S|=q+\sqrt{q}+2$, and $q\ge 9$ is of square order, then
 there exists a point $x\in S$ such that $S\setminus\{x\}$ is the point set
 of a Baer subplane.\qed
\end{lemma}

Let $\cU\subseteq\cL$  be a set of lines. A set $C\subseteq\cP$ is called a
{\em near-blocker of\/} $\cU$ if it meets exactly all but one member of~$\cU$.

\begin{lemma}\label{l:near} Let $\cU$\/ be a set of lines in $PG(2,q)$.
 \begin{lst}
  \item[(a)] Suppose that $\cap_{\ell\in\cU}\ell=\emptyset$.
  Then there exists a near-blocker of size at most $|\cU|/2$.
  \item[(b)] Suppose that $q\ge 5$ is odd and\/ $\cU$ cannot be blocked by a $2$-element set.
  Then there exists a near-blocker of size at most $|\cU|/3+(q+1)/6$.
 \end{lst}
\end{lemma}
\begin{proof}
(a) Let us apply induction on the size of $|\cU|$. The cases $|\cU|=1,2,3$  are trivial.
If $\cU$ cannot be covered by two points then select any point $p\in\cP$ covered
at least twice by the lines of $\cU$ and use induction from $\cU\setminus\cL(\{p\})$.
Otherwise, some two points $x_1,x_2$ cover all lines. Assuming that
$\deg_\cU(x_1)\ge\deg_\cU(x_2)$,
select $x_1$ and one element from all but one of the lines of $\cU$ going
through $x_2$ and avoiding $x_1$.

(b) For $|\cU|\le q+2$ we have $\lfloor|\cU|/2\rfloor\le|\cU|/3+(q+1)/6$ and we can
apply case~(a). (If $|\cU|=q+2$ we make use of the fact that $q$ is odd.)
We may now suppose $|\cU|\ge q+3$, so $\max_p\deg_\cU(p)\ge 3$.
Consider first the case when $\cU$ cannot be covered by three vertices.
Chose a maximum degree vertex $p$ and apply the induction hypothesis
to $\cU\setminus\cL(\{p\})$.
Finally, if some set $\{x_1,x_2,x_3\}$ meets every member of $\cU$ we choose the
two highest degree vertices among them and one element from all but one of
the lines of $\cU$ going through the third, avoiding the other two.
In this way we obtain a near-cover of size at most $2+(|\cU|/3-1)$.
\end{proof}

The following lemma will be useful when $|\cL^e(A)|$, $t_1$, and $t_2$ are all small.

\begin{lemma}\label{l:diff}\ \\[-12pt]
 \begin{lst}
 \item[(a)] Let $A=(\ell\setminus T_1)\cup T_2$ where $\ell$ is a line,
  $T_1\subseteq\ell$, $T_2\cap\ell=\emptyset$, and\/ $t_i=|T_i|$.
  Then $|\cL^e(A)|\ge (t_1+t_2)q-t_2(2t_1+t_2-2)$.
 \item[(b)] Let $A=(B \setminus T_1)\cup T_2$ where $B$ is a Baer subplane,
  $T_1\subseteq B$, $T_2\cap B=\emptyset$, and\/ $t_i=|T_i|$.
  Then $|\cL^e(A)|\ge (t_1+t_2)q-t_2(2t_1+t_2-1)-t_1\sqrt{q}$.
 \end{lst}
\end{lemma}
\begin{proof}
(a) Consider the lines through a point $x\in T_2$. Exactly $q+1-t_1$ of them
meet $\ell\setminus T_1$. At most $t_2-1$ of these lines contain a
further point of $A$ (namely a point from~$T_2$).
Thus we have obtained at least $t_2(q+1-t_1-(t_2-1))$ 2-point lines.
Next consider the $q$ lines through a point $y\in T_1$ other than~$\ell$.
All but $t_2$ avoids $T_2$, too, thus giving at least $t_1(q-t_2)$ zero-point lines.
The total number of these lines gives the desired lower bound.

(b) Every point $x\in T_2$ is incident to at least
$(q-t_1)-(t_2-1)$ 2-point lines, and every point $y\in T_1$ is incident to at least
$q-\sqrt{q}-t_2$ zero-point lines.
\end{proof}

\begin{proof}[Proof of equality in~\eqref{eq:7}.]
Suppose, on the contrary, that we have a set of lines~$R$, $|R|=2q+\sqrt{q}$, such that
for $S=\sum_{\ell\in R}\ell$ we have $|S|<q+\sqrt{q}+2$. Since $|S|$ is even, we have
$|S|\le q+\sqrt {q}$. Since $R$ is odd we have $R=\cL^e(S)$. Thus $S$ meets every line
from $\cL\setminus R$. Let $\cU$ be the set of lines avoiding $S$, we have $\cU\subseteq R$.

First consider the case when there is a set $V$, $|V|\le 2$,
meeting all points of~$\cU$. (This includes the case $\cU=\emptyset$.)
Then $S\cup V$ meets all lines, so is a blocking set.

We claim that $S\cup V$ does not contain a line, so is a non-trivial blocking set.
Suppose, on the contrary, that  there is a line $\ell\subseteq S\cup V$.
Apply \Lm{diff}~(a) with $A=S=(\ell\setminus T_1)\cup T_2$ where
$T_1=\ell\cap V$, $|T_1|\le 2$ and $T_2=S\setminus\ell$, $|T_2|\le|S\cup V|-|\ell|\le\sqrt{q}+1$.
We obtain that
\[
 |\cL^e(S)|\ge t_1q+t_2(q+2-2t_1-t_2)\ge t_1q+t_2(q-\sqrt{q}-3).
\]
Since $|\cL^e(S)|=2q+\sqrt{q}$ we obtain that $|T_1|+|T_2|\le2$ for $q\ge49$.

We finish the proof of our claim by observing that for $|T_1|+|T_2|\le2$, $T_1\subseteq\ell$,
the number of even lines $|\cL^e((\ell\setminus T_1)\cup T_2)|$ cannot be $2q+\sqrt{q}$. % set->number
Indeed, in the case $T_1=\emptyset$ we have $|\cL^e(S)|\le t_2q+2<2q+\sqrt{q}$.
In the case $t_2=0$ we have $|\cL^e(S)|\le 1+t_1q<2q+\sqrt{q}$. % = -> \le
Finally, in the case $t_1=t_2=1$ we have $|\cL^e(S)|=2q-1<2q+\sqrt{q}$.

Consider $S\cup V$, which is a non-trivial blocking set of size at most $ q+\sqrt{q}+2$.
By the Bruen-Thas theorem (\Lm{Bruen}) there is a Baer subplain $B\subseteq S\cup V$.
Thus we know a lot about the structure of $S$, we can write $S=(B\setminus T_1)\cup T_2$
where $T_1= B\setminus S$ (it is a subset of $V$, so $t_1\le 2$) and
$T_2=S\setminus B\subseteq (S\cup V)\setminus B$ so $t_2\le 1$.

We finish the proof of the case $|V|\le 2$ by checking all possible values of $t_1$ and~$t_2$.
In case of $t_1=2$, $t_2=1$, \Lm{diff}~(b) applied to $A=S$ gives $|\cL^e(S)|\ge 3q-4-2\sqrt{q}$.
This exceeds $2q+\sqrt{q}$ for $q\ge25$. We obtain that $t_1+t_2\le2$.
Since $|S|$ is even and $|B|$ is odd their symmetric difference (i.e., $T_1\cup T_2$) is odd, % i->1
we get $t_1+t_2=1$. So $S$ should be one of the examples discussed in the beginning of this
section and we are done.

From now on suppose that there is no set $V$, $|V|\le2$,
meeting all points of~$\cU$. Apply \Lm{near}~(b) to $\cU$ to obtain a near-blocker $C$ of $\cU$
of size at most $|\cU|/3+(q+1)/6$ and a line  $\ell_C\in\cU$ missed by~$C$.
We proceed as in the proof of \Th{f(q+2)}.

The set $S\cup C$ meets all lines except $\ell_C$,
so it is a blocking set of the {\em affine} plane $PG(2,q)\setminus\ell_C$.
Then \Lm{second} yields $|S\cup C|\ge 2q-1$. We obtain
\[
 2q-1\le |S|+|C|\le (q+\sqrt{q})+|\cU|/3+(q+1)/6.
\]
Here $|\cU|\le|R|=2q+\sqrt{q}$ so the right hand side is at most $(11q+8\sqrt{q}+1)/6$. %
This cannot hold for $q\ge 81$. This final contradiction implies that $|S|\le q+\sqrt{q}$
is not possible for $q\ge 81$ and we are done.
\end{proof}

\begin{proof}[Proof of equality in~\eqref{eq:8}.]
This proof is similar to the previous proof, but simpler. Suppose, on the contrary, that we
have a set of lines $R$, $|R|=2q-\sqrt{q}$ such that for $S=\sum_{\ell\in R}\ell$ we
have $|S|<q+\sqrt{q}$. As $|S|$ is even, we have $|S|\le q+\sqrt {q}-2$.
Since $R$ is odd we have $R=\cL^e(S)$. Thus $S$ meets every line from $\cL\setminus R$.
Let $\cU$ be the set of lines avoiding $S$, so that $\cU\subseteq R$.

If there is a set $V$, $|V|\le 2$, meeting all points of $\cU$ (including
the case $\cU=\emptyset$) then $S\cup V$ meets all lines, it is a blocking set
of size at most $q+\sqrt{q}$.
By the Bruen theorem (\Lm{Bruen})  it must contain a line~$\ell$.
Apply \Lm{diff}~(a) with $A=S=(\ell\setminus T_1)\cup T_2$ where
$T_1=\ell\cap V$, $|T_1|\le 2$ and $T_2=S\setminus\ell$, $|T_2|\le|S\cup V|-|\ell|\le\sqrt{q}-1$.
We obtain that
\[
 |\cL^e(S)|\ge t_1q+t_2(q+2-2t_1-t_2)\ge t_1q+t_2(q-\sqrt{q}-1).
\]
Since $|\cL^e(S)|=2q-\sqrt{q}$ we obtain that $|T_1|+|T_2|\le2$ for $q\ge25$.

We finish the investigation of this case by observing that for $|T_1|+|T_2|\le2$,
$T_1\subseteq\ell$, the number of even lines $|\cL^e((\ell\setminus T_1) \cup T_2)|$ %
cannot be $2q-\sqrt{q}$. Since both $S$ and $\ell$ are even sets, their symmetric
difference (i.e., $T_1\cup T_2$) is even.
We have four cases to check according to the value of
$(t_1,t_2)\in\{(2,0),(1,1),(0,2),(0,0)\}$.
The sizes of $|\cL^e(S)|$ are $2q+1$, $2q-1$, again $2q+1$, and $1$, respectively.
None of these is equal to $2q-\sqrt{q}$.

From now on suppose that $\cU\ne\emptyset$ and there is no set $V$, $|V|\le 2$,
meeting all points of~$\cU$. Apply \Lm{near}~(b) to $\cU$ to obtain a near-blocker $C$
of $\cU$ of size at most $|\cU|/3+(q+1)/6$ and a line $\ell_C\in\cU$ missed by~$C$.
We proceed as in the proof of \Th{f(q+2)}.

The set $S\cup C$ meets all lines except~$\ell_C$, so it can be considered as a blocking
set of the affine plane $PG(2,q)\setminus\ell_C$.
Then \Lm{second} yields $|S\cup C|\ge 2q-1$. We obtain
\[
 2q-1\le |S|+|C|\le (q+\sqrt{q}-2)+|\cU|/3+(q+1)/6.
\]
Here $|\cU|\le|R|=2q-\sqrt{q}$ so the right-hand-side is at most $(11q+4\sqrt{q}-11)/6$.
This cannot hold for $q\ge 49$ implying that $|S|\le q+\sqrt{q}$
is not possible for $q\ge 49$ and we are done.
\end{proof}

With some more work we  can see that only the examples from the Baer subplane give
equalities in~\eqref{eq:7} and~\eqref{eq:8} (for $q> q_0$).

\bigskip

Many questions remain open. What is $f(q+2)$, and $f(q+3)$?
The least we should be able to do is to prove better
bounds on these. Also, any information about $f(r)$ for $r\le 2q^{3/2}$
would be great.

\section{Acknowledgements}

The authors are indebted to the referees for their helpful comments
 and suggestions.  %%% ???

\newpage
\section*{Appendix A. Values of $f(r)$ for small $q$.}

\begin{table}[h]
\caption{$q=3$}\vskip-5pt
\[\begin{array}{rlrl}
r&f(r)&r&f(r)\\\hline
1&4 &4&4\\
2&6 &5&4\\
3&6 &6&2
\end{array}\]
\vskip-10pt
\end{table}

\begin{table}[h]
\caption{$q=5$}\vskip-5pt
\[\begin{array}{rlrlrl}
r&f(r)&r&f(r)&r&f(r)\\\hline
1&6 &6&6&11&4\\
2&10&7&8&12&4\\
3&12&8&8&13&6\\
4&12&9&6&14&6\\
5&10&10&2&15&4
\end{array}\]
\vskip-10pt
\end{table}

\begin{table}[h]
\caption{$q=7$}\vskip-5pt
\[\begin{array}{rlrlrlrl}
r&f(r)&r&f(r)&r&f(r)&r&f(r)\\\hline
1&8 &8 &8 &15&6&22&6\\
2&14&9 &12&16&8&23&6\\
3&18&10&10&17&8&24&4\\
4&20&11&10&18&6&25&8\\
5&20&12&12&19&10&26&6\\
6&18&13&8 &20&4&27&6\\
7&14&14&2&21&8&28&4
\end{array}\]
\vskip-10pt
\end{table}

\newcommand\q{\hbox{--}}

\begin{table}[h]
\caption{$q=9$}\vskip-5pt
\[\begin{array}{rlrlrlrlrl}
r&f(r)&r&f(r)&r&f(r)&r&f(r)\\\hline
1&10&10&10&19&8 &28&4 &37&6 \\
2&18&11&16&20&12&29&10&38&6 \\
3&24&12&12&21&10&30&6 &39&8 \\
4&28&13&14&22&10&31&8 &40&8 \\
5&30&14&14&23&12&32&4 &41&10\\
6&30&15&12&24&8 &33&10&42&6 \\
7&28&16&16&25&10&34&6 &43&8 \\
8&24&17&10&26&10&35&8 &44&8 \\
9&18&18&2 &27&12&36&4 &45&6
\end{array}\]
\vskip-10pt
\end{table}
% PG [OK]
%0    0    1  10   2  18   3  24  4   28   5  30   6  30   7  28
%8   24    9  18  10  10  11  16  12  12  13  14  14  14  15  12
%16  16   17  10  18   2  19   8  20  12  21  10  22  10  23  12
%24   8   25  10  26  10  27  12  28   4  29  10  30   6  31   8
%32   4   33  10  34   6  35   8  36   4  37   6  38   6  39   8
%40   8   41  10  42   6  43   8  44   8  45   6

% (dual) Hall [OK if <=14]
%0    0    1  10   2  18   3  24    4 28    5 30   6  30   7 28
%8   24    9  18  10  10  11  20?  12 20?  13 18? 14  14  15 16*
%16  20?  17  10  18   2  19   8   20 12   21 10  22  10  23 12
%24   8   25  10  26  10  27  12   28  4   29 10  30   6  31  8
%32   4   33  10  34   6  35   8   36  4   37  6  38   6  39  8
%40   8   41  10  42   6  43   8   44  8   45  6

% (not) dual Hall [OK if <=14]
%0    0    1 10    2 18    3 24    4 28    5 30    6 30    7 28
%8   24    9 18   10 10   11 91*  12 91*  13 91*  14 91*  15 16*
%16  16*  17 10   18  2   19  8   20 12   21 10   22 10   23 12
%24   8   25 10   26 10   27 12   28 4    29 10   30 6    31 8
%32   4   33 10   34  6   35  8   36 4    37 6    38 6    39 8
%40   8   41 10   42  6   43  8   44 8    45 6

% Hughes [OK if <=14]
%0   0     1 10    2 18    3 24    4 28    5 30    6 30    7 28
%8  24     9 18   10 10   11 20?  12 20*  13 18*  14 18*  15 16*
%16 16*   17 10   18  2   19  8   20 12   21 10   22 10   23 12
%24  8    25 10   26 10   27 12   28  4   29 10   30  6   31  8
%32  4    33 10   34  6   35  8   36  4    37 6   38  6   39  8
%40  8    41 10   42  6   43  8   44  8    45 6

\begin{table}
\caption{$q=11$}\vskip-5pt
\[\begin{array}{rlrlrlrlrlrl}
r&f(r)&r&f(r)&r&f(r)&r&f(r)&r&f(r)&r&f(r)\\\hline
 1&12 &12&12    &23&10    &34&10 &45&8 &56&8 \\
 2&22 &13&20    &24&16    &35&14 &46&6 &57&8 \\
 3&30 &14&14\q26&25&16    &36&4  &47&10&58&6 \\
 4&36 &15&14\q18&26&14    &37&12 &48&8 &59&10\\
 5&40 &16&16    &27&14    &38&10 &49&12&60&8 \\
 6&42 &17&16    &28&12    &39&10 &50&6 &61&8 \\
 7&42 &18&14\q18&29&16    &40&4  &51&10&62&10\\
 8&40 &19&14\q26&30&10    &41&12 &52&8 &63&10\\
 9&36 &20&16\q20&31&14\q18&42&6  &53&12&64&8 \\
10&30 &21&12    &32&12    &43&14 &54&6 &65&8 \\
11&22 &22& 2    &33&16    &44&4  &55&10&66&6
\end{array}\]
\vskip-10pt
\end{table}

\begin{figure}
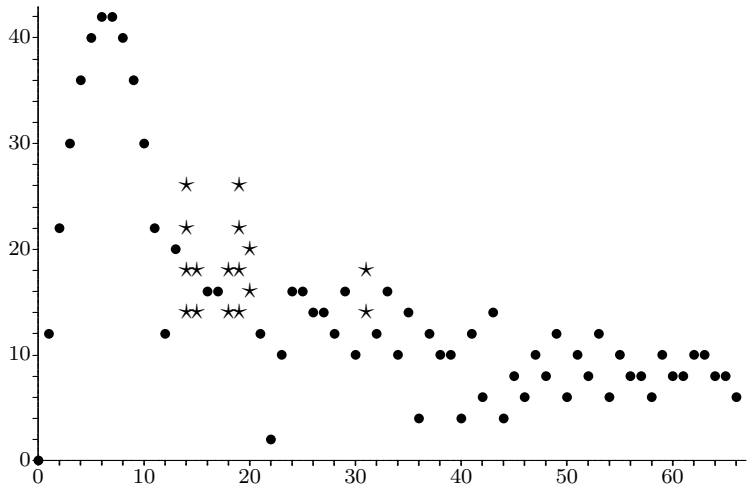

\[
 \unit=4pt
 \setbox\dotb\clap{$\thnline\unit=2pt\dl000{-1}$}
 \xoff=0pt\yoff=0pt\dline{-2}{0}{68}{0}{35}
 \setbox\dotb\llap{$\thnline\unit=2pt\dl00{-1}0$}
 \xoff=0pt\yoff=0pt\dline{0}{-2}{0}{44}{23}
 \thnline\dl{0}{0}{67}{0}\dl{0}{0}{0}{43}
 \ptld{0}{0}{0}\ptld{10}{0}{10}\ptld{20}{0}{20}\ptld{30}{0}{30}
 \ptld{40}{0}{40}\ptld{50}{0}{50}\ptld{60}{0}{60}
 \ptll{0}{0}{0}\ptll{0}{10}{10}\ptll{0}{20}{20}\ptll{0}{30}{30}
 \ptll{0}{40}{40}
 \pt{0}{0}\pt{1}{12}\pt{2}{22}\pt{3}{30}\pt{4}{36}\pt{5}{40}
 \pt{6}{42}\pt{7}{42}\pt{8}{40}\pt{9}{36}\pt{10}{30}\pt{11}{22}
 \pt{12}{12}\pt{13}{20}\px{14}{14}\px{14}{18}\px{14}{22}\px{14}{26}
 \px{15}{14}\px{15}{18}\pt{16}{16}\pt{17}{16}\px{18}{14}\px{18}{18}
 \px{19}{14}\px{19}{18}\px{19}{22}\px{19}{26}\px{20}{16}\px{20}{20}
 \pt{21}{12}\pt{22}{2}\pt{23}{10}\pt{24}{16}\pt{25}{16}\pt{26}{14}
 \pt{27}{14}\pt{28}{12}\pt{29}{16}\pt{30}{10}\px{31}{14}\px{31}{18}
 \pt{32}{12}\pt{33}{16}\pt{34}{10}\pt{35}{14}\pt{36}{4}\pt{37}{12}
 \pt{38}{10}\pt{39}{10}\pt{40}{4}\pt{41}{12}\pt{42}{6}\pt{43}{14}
 \pt{44}{4}\pt{45}{8}\pt{46}{6}\pt{47}{10}\pt{48}{8}\pt{49}{12}
 \pt{50}{6}\pt{51}{10}\pt{52}{8}\pt{53}{12}\pt{54}{6}\pt{55}{10}
 \pt{56}{8}\pt{57}{8}\pt{58}{6}\pt{59}{10}\pt{60}{8}\pt{61}{8}
 \pt{62}{10}\pt{63}{10}\pt{64}{8}\pt{65}{8}\pt{66}{6}
 \hskip66\unit
\]
\caption{Graph of $f(r)$ for $q=11$. Dots represent known values, and stars represent possible
values for the values of $r$ for which $f(r)$ is unknown.}
\end{figure}

\end{document}